\documentclass[twoside]{amsart}
\usepackage{latexsym}
\usepackage{amssymb,amsmath,amsopn}
\usepackage[dvips]{graphicx}   
\usepackage{color,epsfig}      

\newtheorem{thm}{Theorem}
\newtheorem{lem}{Lemma}
\theoremstyle{definition}

\newtheorem{ques}{Question}

\renewcommand{\Re}{\mathbb R}

\newcommand{\F}{\mathfrak{F}}
\renewcommand{\S}{\mathbb{S}}

\DeclareMathOperator{\inter}{int}
\DeclareMathOperator{\bd}{bd}
\DeclareMathOperator{\dist}{dist}
\DeclareMathOperator{\conv}{conv}
\DeclareMathOperator{\card}{card}

\begin{document}
\title[On the Hadwiger number of starlike disks]{On the Hadwiger numbers of starlike disks}

\author[Z. L\'angi]{Zsolt L\'angi}
                                                                               
\address{Zsolt L\'angi, Dept. of Geometry, Budapest University of Technology,
Budapest, Egry J\'ozsef u. 1., Hungary, 1111}
\email{zlangi@math.bme.hu}
                                                                        

\begin{abstract}
The Hadwiger number $H(J)$ of a topological disk $J$ in $\Re^2$
is the maximal number of pairwise nonoverlapping translates of $J$ that touch $J$.
It is well known that for a convex disk, this number is six or eight.
A conjecture of A. Bezdek., K. and W. Kuperberg says that the Hadwiger number of a starlike
disk is at most eight.
A. Bezdek proved that this number is at most seventy five for any starlike disk.
In this note, we prove that the Hadwiger number of a starlike disk is at most thirty five.
Furthermore, we show that the Hadwiger number of a topological disk $J$ such that
$(\conv J) \setminus J$ is connected, is six or eight.
\end{abstract}
\maketitle

\section{Introduction and Preliminaries}

This paper deals with topological disks in the Euclidean plane $\Re^2$.
We make use of the linear structure of $\Re^2$, and identify a point with its position vector.
We denote the origin by $o$, and the standard orthonormal basis of $\Re^2$ by $\{ e_x, e_y \}$.
For simplicity, we use the notation $(\alpha,\beta) = \alpha e_x + \beta e_y \in \Re^2$ for any $\alpha, \beta \in \Re$.
For a set $X \subset \Re^2$, $\conv X$, $\card X$, $\inter X$ and $\bd X$ denote
the \emph{convex hull}, the \emph{cardinality}, the \emph{interior} and the \emph{boundary}
of $X$, respectively.
If $X= \{ x_1, x_2, \ldots, x_k \}$ is a finite set, we may use the notation $\conv X = [x_1,x_2,\ldots, x_k]$.
In particular, for $p,q \in \Re^2$, the closed segment with endpoints $p$ and $q$ is denoted by $[p,q]$.
We set $(p,q)=[p,q] \setminus \{ p,q \}$, $(p,q]=[p,q]\setminus \{p\}$ and $[p,q) = [p,q] \setminus \{ q\}$.
The Euclidean norm of a point $p \in \Re^2$ is denoted by $|| p ||$.

A \emph{topological disk}, or shortly a \emph{disk}, is a compact subset of $\Re^2$ with
a simple, closed, continuous curve as its boundary.
In other words, a disk is a subset of $\Re^2$ homeomorphic to the closed unit disk of the plane.
Two disks $J_1$ and $J_2$ are \emph{nonoverlapping}, if their interiors are disjoint.
If $J_1$ and $J_2$ are nonoverlapping and $J_1 \cap J_2 \neq \emptyset$, then
$J_1$ and $J_2$ \emph{touch}.
A disk $S$ is \emph{starlike relative to a point $p$}, if, for every $q \in S$,
$S$ contains the closed segment $[p,q]$.
In particular, a convex disk $K$ is starlike relative to any point $p \in K$.

The \emph{Hadwiger number}, or \emph{translative kissing number}, of
a disk $J$, denoted by $H(J)$, is the maximal number of pairwise nonoverlapping translates of $J$
that touch $J$.
Gr\"unbaum \cite{G61} proved that the Hadwiger number of a parallelogram is eight, and that
the Hadwiger number of any other convex disk is six.
In \cite{HLS59}, the authors showed that the Hadwiger number of any disk is at least six.
A. Bezdek, K. and W. Kuperberg \cite{BKK95} asked if $H(J) \leq 8$ for any disk $J$, or if it is not so,
if there is a universal constant $\kappa \in \Re$ such that $H(J) \leq \kappa$ for every disk $J$
(see also Problem 5, p. 95 in the book \cite{BMP05}).
They formulated also the conjecture that $H(S) \leq 8$ for any starlike disk $S$.

In 2007, Cheong and Lee \cite{CL07} constructed, for every $n>0$, a disk with Hadwiger number at least $n$,
and thus showed that the answer for the question mentioned above is no.
On the other hand, A. Bezdek proved in \cite{B97} that the Hadwiger number of a starlike disk is at most seventy five.
In \cite{L08} it is shown that the Hadwiger number of a centrally symmetric starlike disk is at most twelve.
In the first part of the paper we prove the following theorem.

\begin{thm}\label{thm:starlike}
The Hadwiger number $H(S)$ of a starlike disk $S$ is at most thirty five.
\end{thm}

If a set $S$ is the union of finitely many closed segments meeting at a given point, let us
call it a \emph{starlike} set of segments.
We may define the Hadwiger number of a starlike set $S$ of segments as the maximal cardinality
of a family of translates of $S$ such that each translate contains a point of $S$
and no translate crosses $S$ or any other translate in the family.
Clearly, any upper bound on the set of the Hadwiger numbers
of starlike sets of segments is an upper bound on the set of the Hadwiger numbers
of starlike disks.
We note that the estimate seventy five of A. Bezdek holds also for starlike sets of segments
(cf. Theorem 2 in \cite{B97}).
Unfortunately, our proof cannot be generalized for starlike sets of segments.
Figure~\ref{fig:starshaped} shows that the assertion in Lemma~\ref{lem:convexity}
fails if $S$ is not a disk.
This figure shows also the existence of a starlike set $S$ of segments with $H(S) > 8$.
In the author's knowledge, the configuration in Figure~\ref{fig:starshaped}
was found independently by K. Swanepoel and P. Papez.

\begin{figure}[ht]
\includegraphics[width=0.5\textwidth]{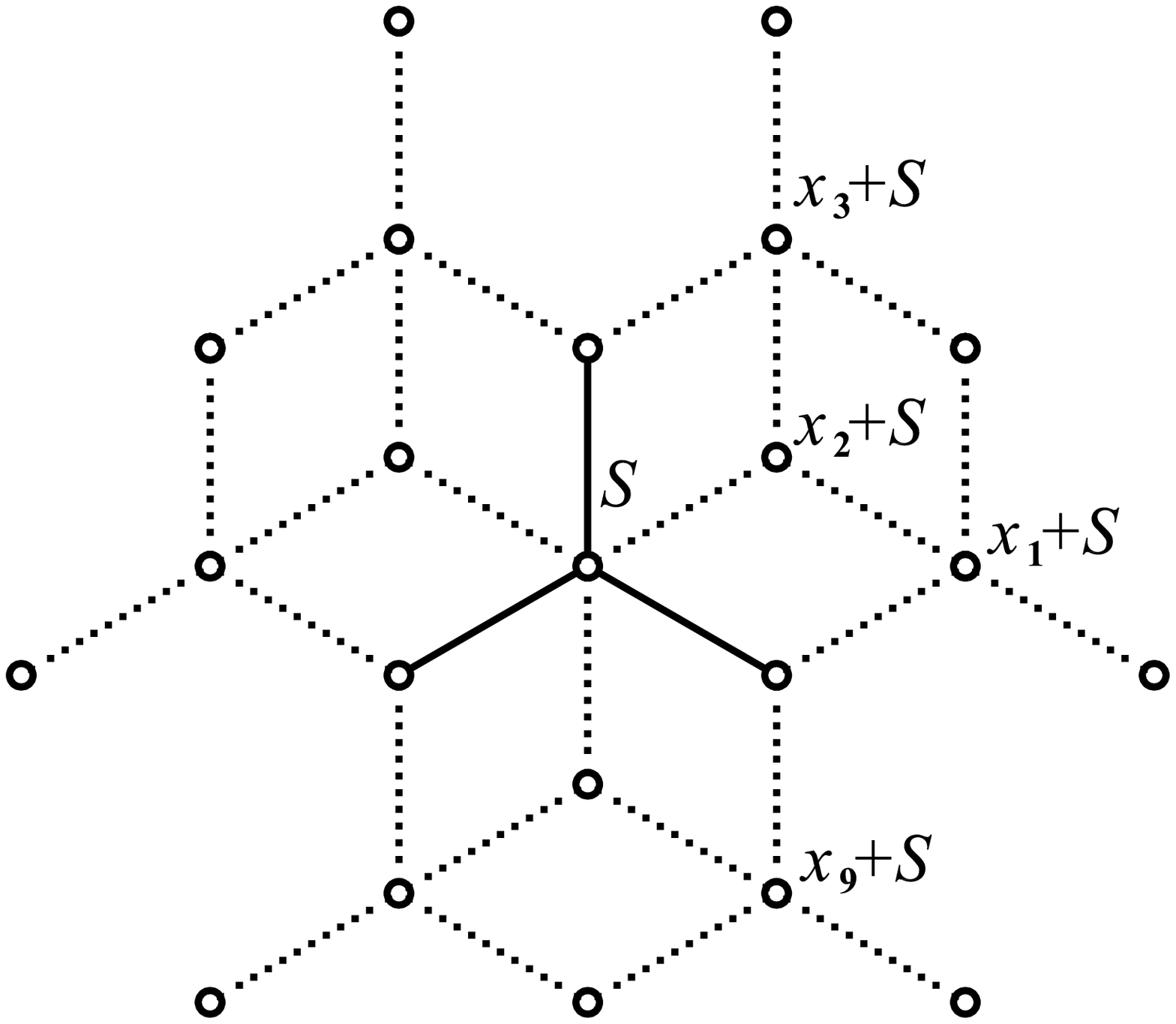}
\caption[]{}
\label{fig:starshaped}
\end{figure}

We ask the following question.

\begin{ques}
Is it true that the Hadwiger number of any starlike set of segments
is at most nine?
\end{ques}

In the second part we examine disks that are not necessarily starlike.
For a disk $J$, we call the connected components of $(\conv J) \setminus J$
the \emph{pockets} of $J$.
Clearly, a disk is convex if and only if it has no pockets.
Our goal is to characterize the Hadwiger numbers of disks with at most one pocket.

\begin{thm}\label{thm:pocket}
Let $J$ be a disk with at most one pocket.\\
{\bf \ref{thm:pocket}.1} If there is a direction $u \in \S^1$ such that
the intersections of $J$, with the two supporting lines of $\conv J$ parallel to $u$, are two segments
of the same length $\lambda > 0$, and $\lambda u+J$ touches $J$, then $H(J) = 8$.\\
{\bf \ref{thm:pocket}.2} Otherwise, $H(J)=6$ (cf. Figure~\ref{fig1}).
\end{thm}

\begin{figure}[ht]
\includegraphics[width=0.9\textwidth]{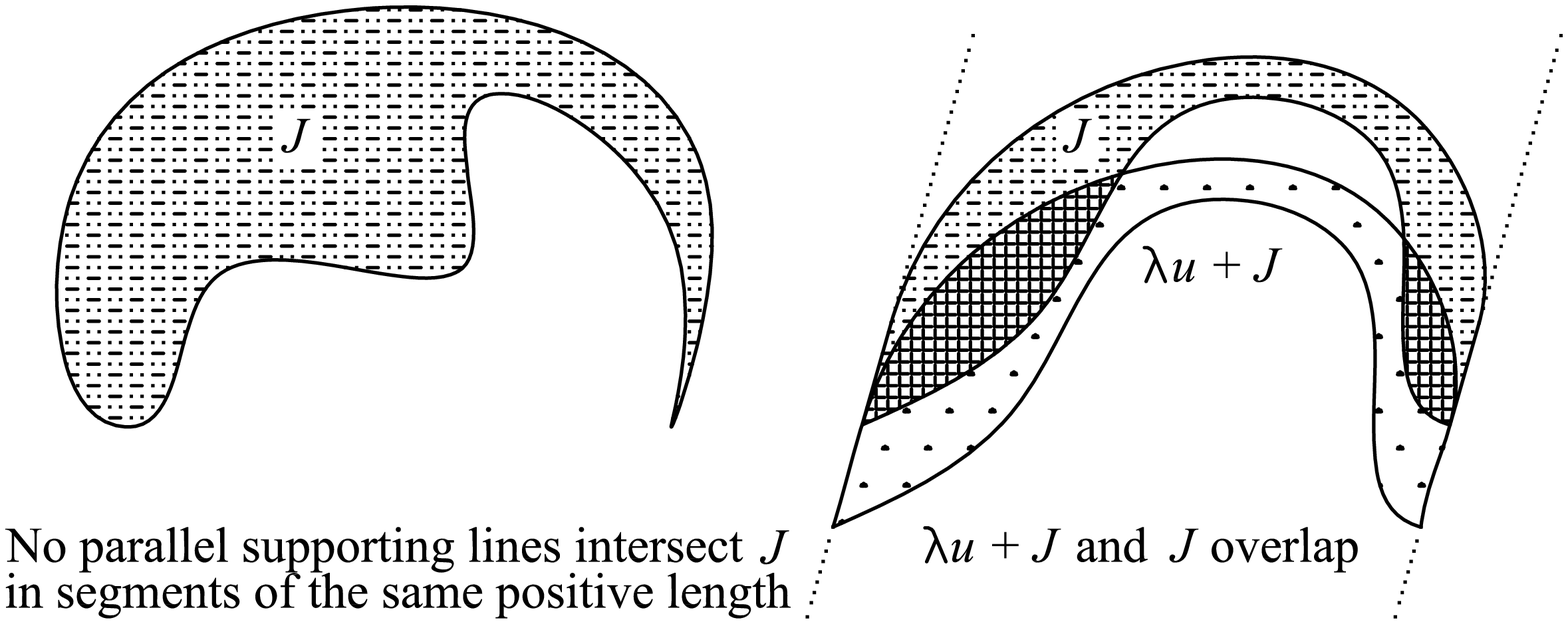}
\caption[]{}
\label{fig1}
\end{figure}

We call a disk $J$ that satisfies the conditions in \ref{thm:pocket}.1 a \emph{parallelogram-like disk} (cf. Figure~\ref{fig2}).

\begin{figure}[ht]
\includegraphics[width=0.9\textwidth]{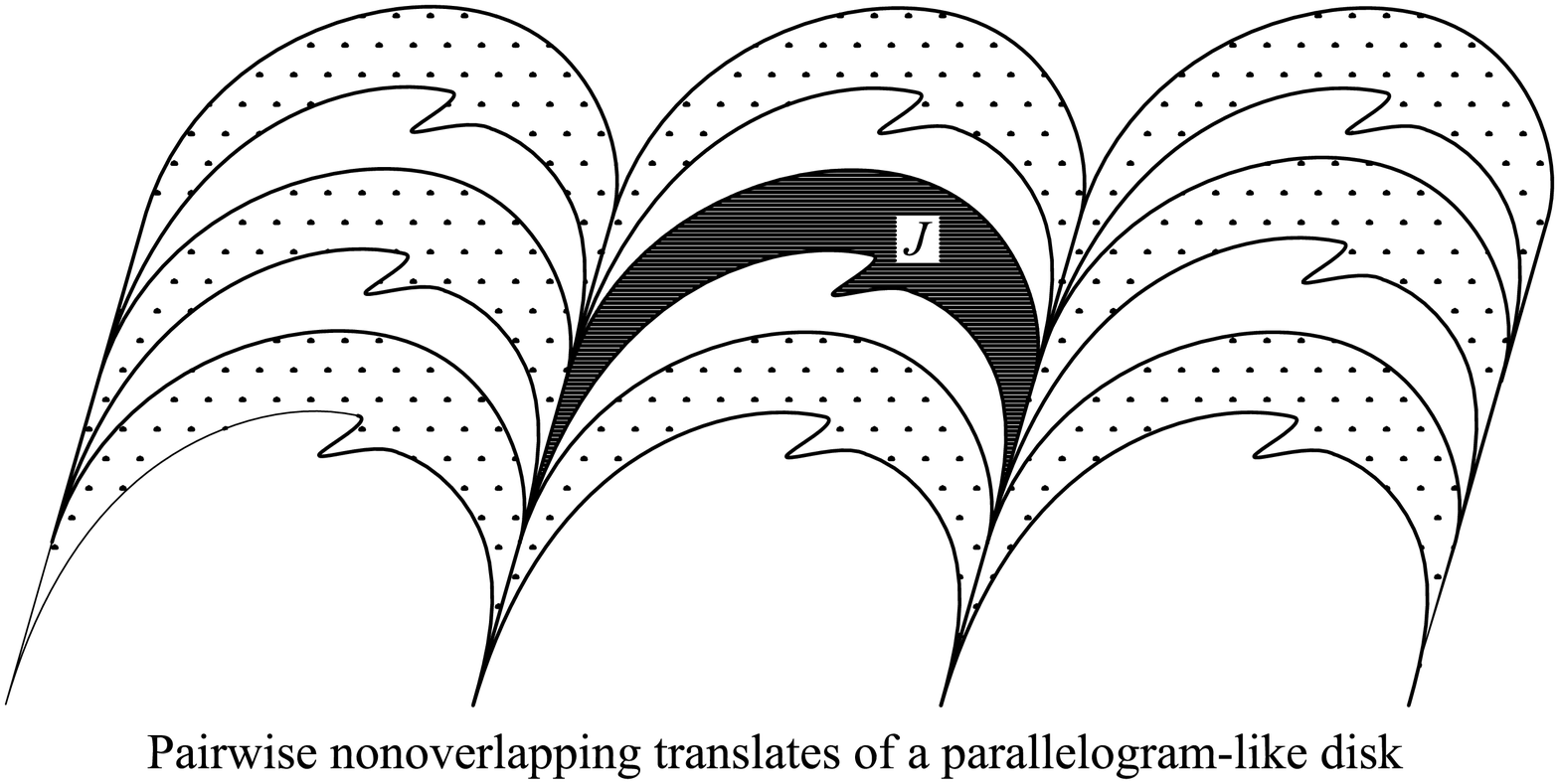}
\caption[]{}
\label{fig2}
\end{figure}

We note that the disk $D_n^m$ ($m \geq n$) in \cite{CL07}, with $n$ pairwise nonoverlapping translates touching $D_n^m$, has $2n+2$ pockets.
With reference to this observation and Theorem~\ref{thm:pocket}, we ask the following question.

\begin{ques}
Is it true for every positive integer $k$, that there is an integer $N(k)$ such that the Hadwiger number of a disk,
with at most $k$ pockets, is at most $N(k)$?
\end{ques}

\section{Proof of Theorem~\ref{thm:starlike}}\label{sec:proofstarlike}

Let $S \subset \Re^2$ be a disk that is starlike relative to the origin, and let $\F = \{ S_i : i=1,2,\ldots, n\}$
be a family of pairwise nonoverlapping translates of $S$, with $n=H(S)$, such that each $S_i = x_i +S$ touches $S$.
Let $K=\conv S$,  $K_i = \conv S_i$ for $i=1,2,\ldots, n$,
$X= \{ x_i: i=1,2,\ldots, n\}$, and $C = \conv X$.
Furthermore, let $R_i= \{ \lambda x_i: \lambda \in \Re \hbox{ and } \lambda \geq 0 \}$. 

First, we prove a few lemmas that we use in the proof of Theorem~\ref{thm:starlike}.

\begin{lem}\label{lem:convexity}
We have $o \in \inter C$, and $X \subset \bd C$.
\end{lem}

\begin{proof}
Note that if $o \in \bd \conv \left( X \cup \{ o \} \right)$, then there is a supporting line $\bar{L}$ of
$F = \conv \left( S \cup \left( \bigcup_{i=1}^n S_i \right) \right)$ that passes through a point of $S$.
Thus, there is a translate of $S$, on the other side of $\bar{L}$, that touches $S$ and does not overlap $F$.
Since $n = H(S)$, we have a contradiction, which proves the first statement.

For contradiction, suppose that $x_i \in \inter C$ for some value of $i$.
Note that if $i \neq j$, then $x_i \notin [o,x_j]$.
Thus, there are indices $j \neq k$ such that $x_i \in \inter [o, x_j, x_k]$.
Since $H(S')=H(S)$ for any affine image $S'$ of $S$, we may assume that
$x_j=e_x$ and $x_k=e_y$.

Consider points $p \in S_j \cap S$ and $q \in S_k \cap S$, and note that $[o,p], [o,q], [o,p-x_j], [o,q-x_k] \subset S$.
Our aim is to show that for any such starlike disk $S$, $S_i$ overlaps $S$, $S_j$ or $S_k$.
In our examination, to help the reader follow the arguments, the segments in the figures belonging to 
$S$, $S_j$ or $S_k$ are drawn with continuous lines, and all the other lines are dotted or dashed.

Observe that $x_i$ is not contained in the open parallelograms $P_j = \inter [o, x_j, p, x_j-p]$,
as otherwise the segment $[x_i,x_i+p]$ crosses $[x_j,p]$, and thus, $S_i$ and $S_j$ overlap
(note that this argument is valid also in the case that $p \in [o,x_j]$).
Similarly, $x_i$ is not contained in $P_k = \inter [o,x_k,q,x_k-q]$, since otherwise $S_i$ and $S_k$ overlap.
We set $T = [o,x_j,x_k]$ and $Q= (\inter T) \setminus \left( P_j \cup P_k \right)$.
So far, we have that $x_i \in Q$.

\begin{figure}[ht]
\includegraphics[width=0.6\textwidth]{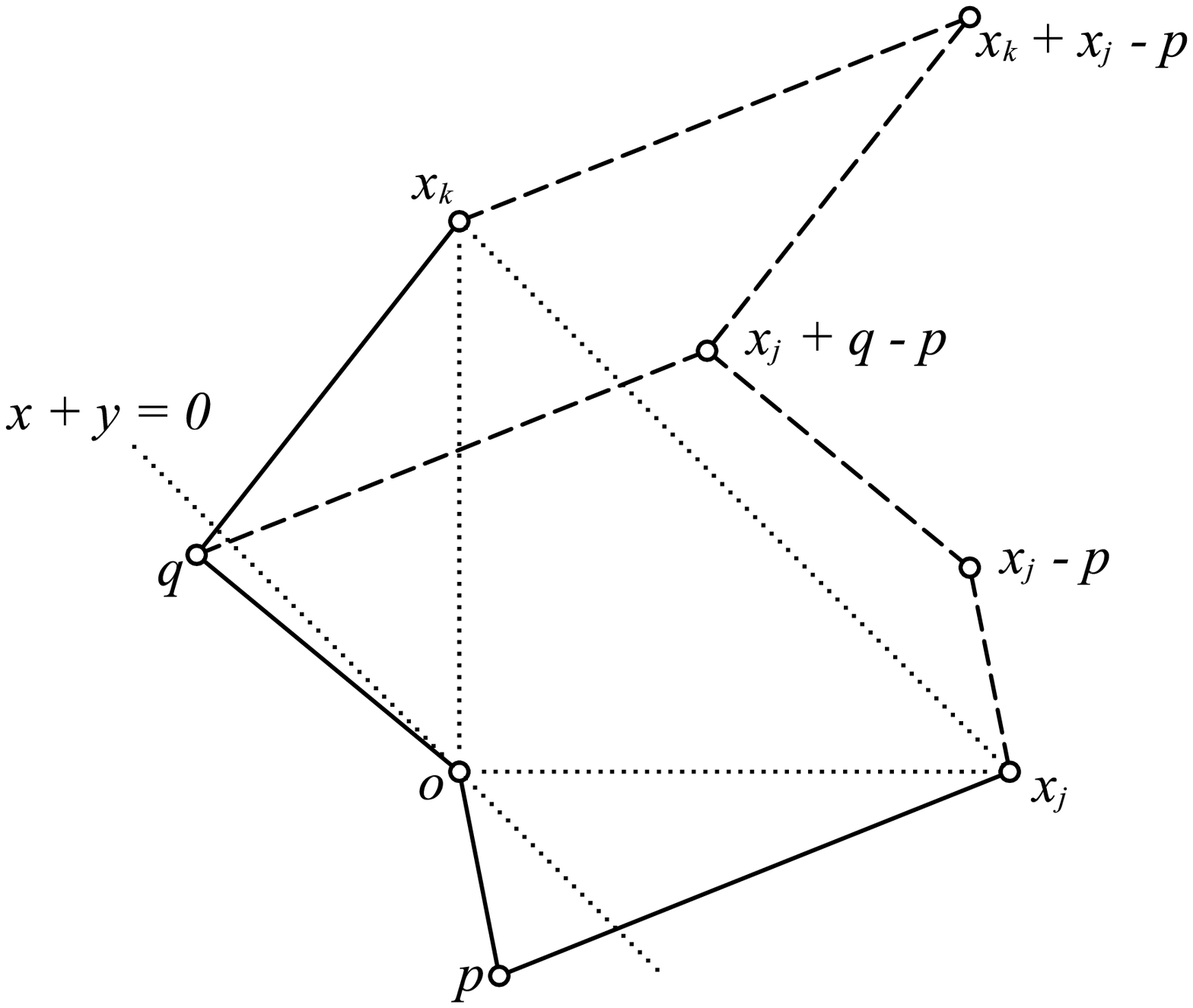}
\caption[]{}
\label{fig3}
\end{figure}

Let $f:\Re^2 \to \Re$ be defined by $f((\alpha, \beta)) = \alpha + \beta$.
We show that $0 \leq f(p) \leq 1$ and $0 \leq f(q) \leq 1$.

For contradiction, suppose first that $f(p) < 0$ or $f(q) < 0$.
Without loss of generality, we may assume that $f(p) < 0$ and $f(p) \leq f(q)$ (cf. Figure~\ref{fig3}), which yields
that $Q \subseteq [o,x_j-p) + \left( [o,q] \cup [q,x_k] \right)$.
If $x_i \in \left( [o,x_j-p) + [o,q) \right)$, then $[x_i,x_i+x_j-p]$ crosses
$[o,q]$, and thus, $S_i$ overlaps $S$; a contradiction.
Similarly, if $x_i \in \left( [o,x_j-p) + (q,x_k] \right)$, then $[x_i,x_i+x_j-p]$ crosses
$[q,x_k]$, and $S_i$ overlaps $S_k$.
Finally, if $x_i \in [q,q+x_j-p)$, then $q$ lies in the relative interior of a segment in $S_i$,
from which it readily follows that $S_i$ is not a disk; a contradiction.

Next, suppose that $f(p) > 1$ or $f(q) > 1$.
Without loss of generality, we may assume that $f(p) > 1$ and that $0 \leq f(q) \leq f(p)$.
Then $Q \subseteq [o,p) + \left( [o,x_k-q] \cup [x_k-q,x_k] \right)$.
From here, the assertion follows by an argument similar to the one in the previous paragraph.

In the following, we denote the line with equation $x+y=1$ by $L$.

\emph{Case 1}, both the $y$-coordinate of $p$ and the $x$-coordinate of $q$ are negative.
Without loss of generality, we may assume that $f(q) \geq f(p)$.
Then, since $f$ is linear, we have that $f(x_j+q-p) \geq 1$, or in other words, that
$L$ separates $x_j+q-p$ from the origin.
Note that $Q$ is covered by the union of the sets $U_1=[x_j,x_j-p) + [o,q)$,
$U_2 = [o,q) + [o,x_j-p)$, $U_3 = [x_k,q) + [o,x_j-p)$, $[q,x_j+q-p)$ and $[x_j-p,x_j+q-p)$
(cf. the left-hand side of Figure~\ref{fig:lemma1}).
If $x_i \in U_1$, then $[x_i,x_i+p]$ and $[x_j,x_j+q]$ cross, and thus, $S_i$ and $S_j$ overlap; a contradiction.
If $x_i \in U_2$ or $x_i \in U_3$, then $[x_i,x_i+x_j-p]$ crosses $[o,q]$ or $[q,x_k]$, respectively, and thus, $S_i$ overlaps $S$ or $S_k$.
If $x_i \in [q,x_j+q-p)$, then $S$ and $S_k$ touch each other in a relative interior point of $[x_i,x_i-p]$,
which yields that $S_i$ is not a disk; a contradiction.
Finally, if $x_i \in [x_j-p,x_j+q-p)$, then $S_i$ meets the segments $[o,q)$ and $[x_j,x_j+q)$ from different sides.
Since $[x_j,x_j+q)$ is the translate of $[o,q)$ in $S_j$, from this it follows that $S$ is not a disk; a contradiction.

\begin{figure}[ht]
\includegraphics[width=\textwidth]{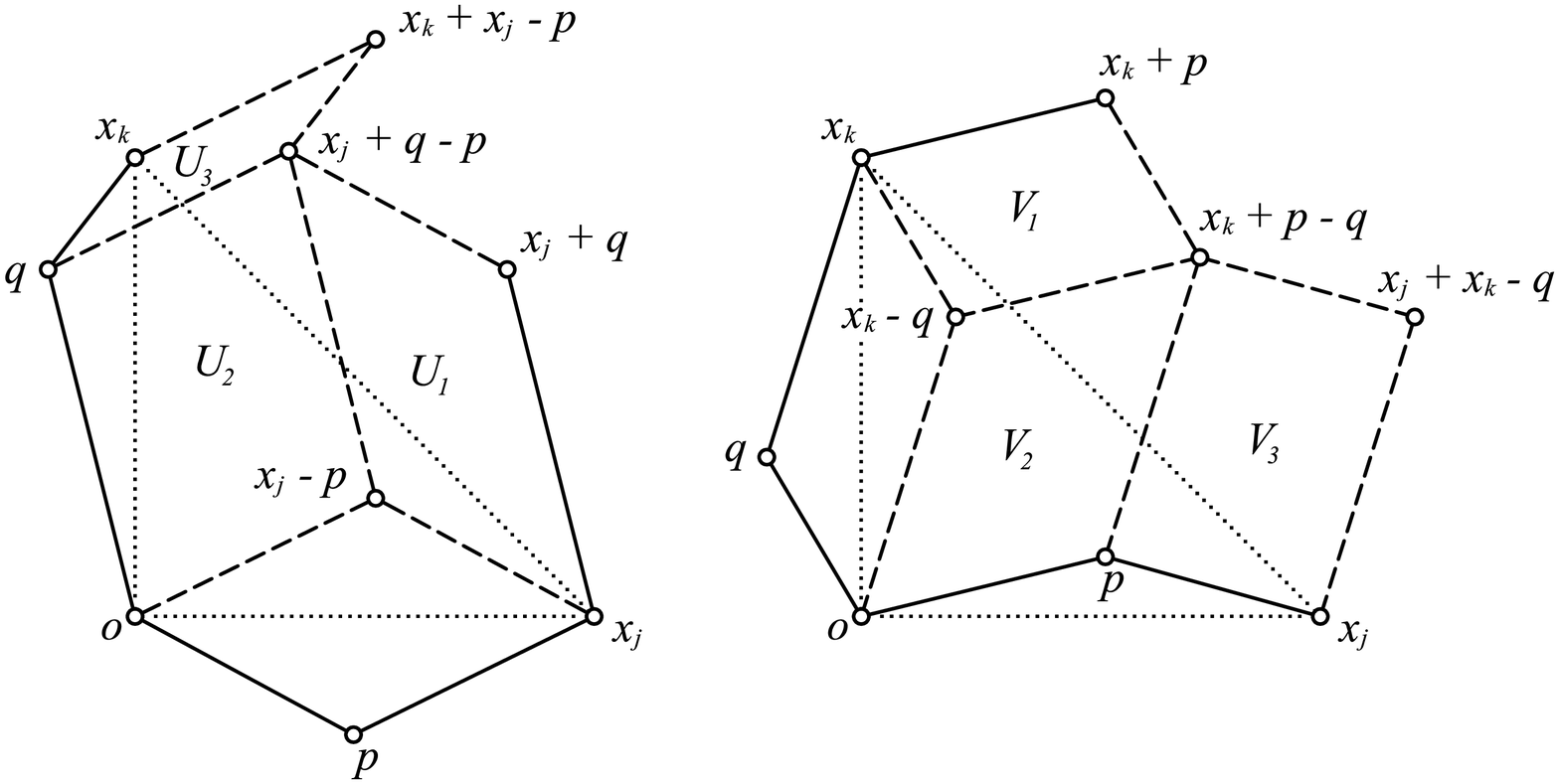}
\caption[]{}
\label{fig:lemma1}
\end{figure}

\emph{Case 2}, either the $y$-coordinate of $p$ or the $x$-coordinate of $q$ is negative.
Without loss of generality, we may assume that the $y$-coordinate of $p$ is nonnegative and
that the $x$-coordinate of $q$ is negative.
First, we examine the case that $f(p) \geq f(q)$, which yields that $L$ separates $o$ and $x_k+p-q$
(cf. the right-hand side of Figure~\ref{fig:lemma1}).
Then $Q$ is covered by the union of the sets $V_1 = [x_k,x_k-q) + [o,p)$, $V_2 = [o,x_k-q) + [o,p)$,
$V_3 = [x_j,p) + [o,x_k-q)$, $[p,x_k+p-q)$ and $[x_k-q,x_j+p-q)$.
If $x_i \in V_1$, $x_i \in V_2$ or $x_i \in V_3$, then $S_i$ overlaps $S_k$, $S$ or $S_j$, respectively.
If $x_i \in [p,x_k+p-q)$ or $x_i \in [x_k-q,x_j+p-q)$, then $S$ is not a disk.

If $f(p) \leq f(q)$, then the assertion follows by a similar argument.

\emph{Case 3}, both the $y$-coordinate of $p$ and the $x$-coordinate of $q$ are nonnegative.
The proof in this case is similar to the proof in the previous two cases, hence we omit it.
\end{proof}

With reference to Lemma~\ref{lem:convexity}, we may relabel
the indices of the elements of $\F$ in a way that $x_1, x_2, \ldots, x_n=x_0$ are in counterclockwise order
on $\bd C$.

\begin{lem}\label{lem:counterclockwise}
Consider points $w_i \in S \cap S_i$ for $i=1,2,\ldots, n$. Then $w_1, w_2, \ldots, w_n$
are in this counterclockwise order around $o$.
\end{lem}

\begin{proof}
Note that as $o \in \inter C$, and the points $x_1, x_2, \ldots, x_n$ are in this counterclockwise order on $\bd C$,
they are in the same order around $o$.
We define the points $\bar{x}_i$ as follows: If $w_i \in \inter C$, then $\bar{x}_i = x_i$, and otherwise
it is the intersection point of $[o,w_i]$ and $\bd C$.
Let $\bar{R}_i = \{ \lambda \bar{x}_i : \lambda \in \Re \hbox{ and } \lambda \geq 0 \}$,
and let $Q_i = \inter \conv (R_i \cup \bar{R}_i)$.

First, we show that if $x_j \in Q_i$ for some $j \neq i$, then
$x_j \notin [o,w_i,x_i]$, $w_i \in [o,w_j,x_j]$ and $w_j \notin \inter C$.
Consider some $i \neq j$ with $x_j \in Q_i$. Then $w_i \notin \inter C$, as otherwise $R_i = \bar{R}_i$.
If $x_j \in \inter [o,w_i,x_i]$, then $[x_j,x_j+w_i]$ crosses $[x_i,w_i]$, and $S_i$ and $S_j$ overlap; a contradiction.
If $x_j \in (w_i,x_i)$, then $[x_j,x_j+(w_i-x_i)] \subset Sj$, which, since this segment is the translate
of $[x_i,w_i]$ by $x_j-x_i$ and since their relative interiors intersect, yields that $S$
is not a disk; a contradiction.
If $x_j \notin [o,w_i,x_i]$, then $[w_j,x_j] \cap \bar{R}_i \neq \emptyset$, as otherwise $[o,w_j]$ crosses
$[w_i,x_i]$ or $x_i \in \inter [o,w_j,x_j]$ (cf. Figure~\ref{fig4}).
Thus, in this case $w_i \in [o,w_j,x_j]$, which, as $x_j \in \bd C$, yields that $w_j \notin \inter C$.

\begin{figure}[ht]
\includegraphics[width=0.6\textwidth]{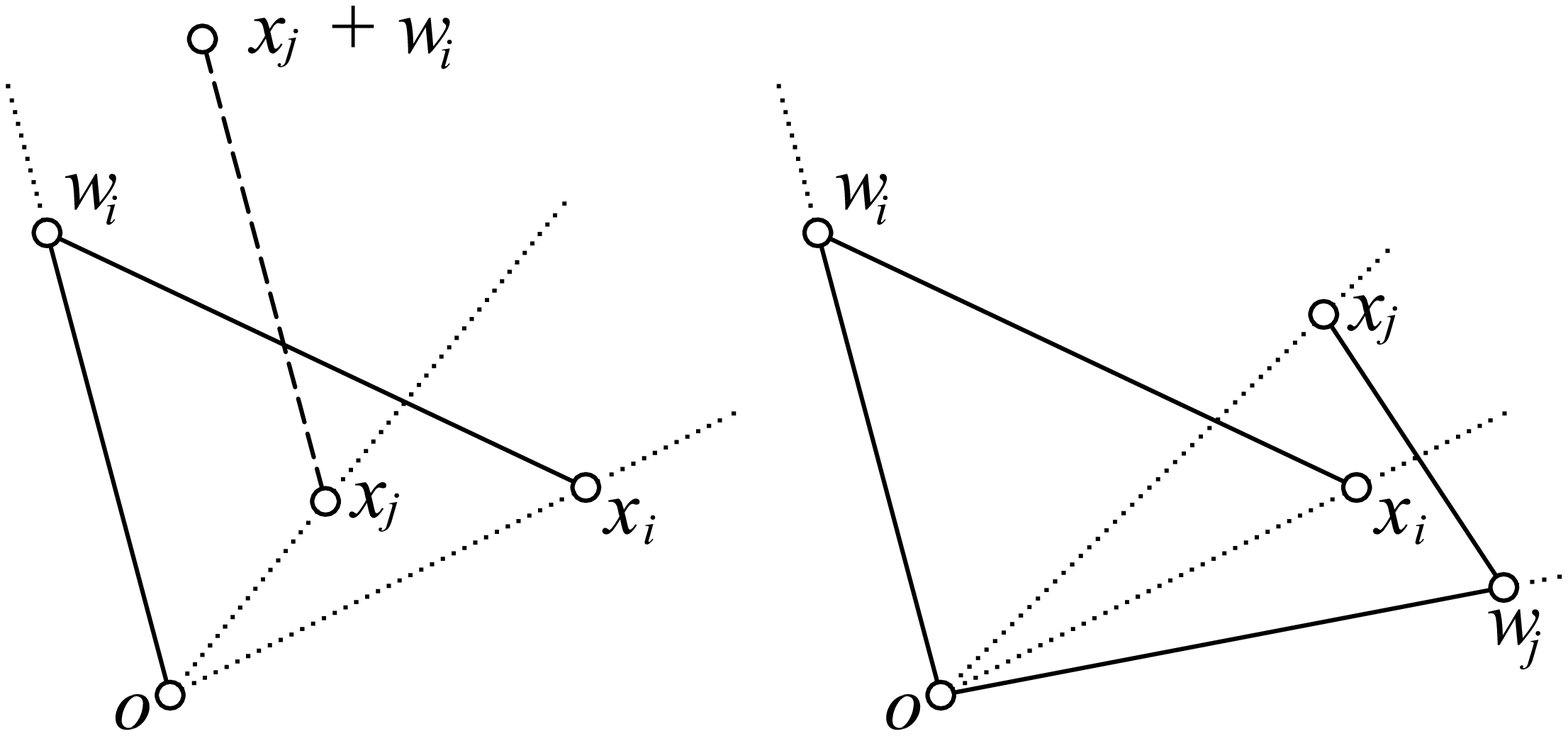}
\caption[]{}
\label{fig4}
\end{figure}

Next, we show that $\bar{x}_1, \bar{x}_2, \ldots, \bar{x}_n$ are in this counterclockwise order around $o$.
To do this, it suffices to show that there are no values of $i \neq j$ such that
$\bar{x}_i, \bar{x}_j$ and $\bar{x}_{i+1}$ are in this counterclockwise order around $o$.
Suppose for contradiction that there are such values.

First we consider the case that $x_i$, $w_j$ and $x_{i+1}$ are in this counterclockwise order around $o$.
Since $x_i$, $x_j$ and $x_{i+1}$ are \emph{not} in this counterclockwise order, we have
$x_i \in Q_j$ or $x_{i+1} \in Q_j$, say $x_i \in Q_j$.
Then, clearly, $x_{i+1} \notin Q_j$, and, by the argument in the second paragraph of this proof,
we have $x_i \notin [o,x_j,w_j]$, $w_j \notin \inter C$ and $w_j \in [o,w_i,x_i]$.
Thus, $w_i \notin \inter C$, which yields that $x_i, w_i$ and $x_{i+1}$ are in this counterclockwise order.
Since $x_j, w_i$ and $w_{i+1}$ are in this counterclockwise order, it follows that
so are $\bar{x}_j, \bar{x}_i$ and $\bar{x}_{i+1}$.

Now we examine the case that $x_i$, $w_j$ and $x_{i+1}$ are not in this counterclockwise order.
Then, since neither are $x_i$, $x_j$ and $x_{i+1}$, we have that $w_j \in Q_i$ or $w_j \in Q_{i+1}$,
say $w_j \in Q_i$. From this, we obtain that $w_j \in [o,w_i,x_i]$, and as $x_j \notin [o,w_i,x_i]$,
we have that $[w_j,x_j]$ intersects both $[o,x_i]$ and $[o,x_{i+1}]$.
Since $x_i,x_j \notin [o,w_{i+1},x_{i+1}]$, this implies that $w_j \in [o,w_{i+1},x_{i+1}]$
and $w_{i+1} \in [o,w_i,x_i]$. From this, it readily follows that $w_i,w_j,w_{i+1} \notin \inter C$,
and thus, that $\bar{x}_i$, $\bar{x}_{i+1}$ and $\bar{x}_j$ are in this counterclockwise order around $o$.

We have shown that $\bar{x}_1, \bar{x}_2, \ldots, \bar{x}_n$ are in this counterclockwise order around $o$.
Since these points are in $\bd C$ and they can be connected to $o$ by mutually noncrossing polygonal curves
in $\inter C$, their counterclockwise order around $o$ is the same as that of the points $w_1, w_2, \ldots, w_n$.
\end{proof}

We need the next lemma of A. Bezdek to prove Lemma~\ref{lem:nonoverlapping} (cf. Lemma 3 in \cite{B97}).

\begin{lem}[A. Bezdek]\label{lem:Bezdek}
For any $i = 1,2, \ldots, n$, $\inter K_i$ contains at most one element of $X \setminus \{ x_i \}$.
\end{lem}

We call $S_i$ and $S_j$ \emph{separated}, if $x_i \notin \inter K_j$, and $x_j \notin \inter K_i$.

\begin{lem}\label{lem:nonoverlapping}
There is a subfamily $\F'$ of $\F$, of cardinality at least $\lfloor \frac{n-2}{2} \rfloor$,
such that any two elements of $\F'$ are separated.
\end{lem}

\begin{proof}
For $i=1,2,\ldots,n$, we choose points $w_i \in S \cap S_i$, and set $\Gamma_i = [o,w_i] \cup [w_i,x_i]$.
By Lemma~\ref{lem:counterclockwise}, the points $w_1, w_2, \ldots, w_n$ are in counterclockwise order around $o$.

By Lemma~\ref{lem:Bezdek}, $\inter K_i$ contains at most one point of $X$ different from $x_i$.
Hence, if $X \cap \inter K_i \subset \{ x_{i-1}, x_i, x_{i+1} \}$ for every value of $i$,
the assertion immediately follows with $\F' = \{ S_{2m}: m=1,2,\ldots, \lfloor n/2 \rfloor \}$.
Thus, it suffices to show that $X \cap \inter K_i \not\subset \{ x_{i-1}, x_i, x_{i+1} \}$
for at most two values of $i$, as in this case, after removing these elements of $\F$, we may choose
the elements of $\F'$ like in the previous case.

Consider the case that $x_j \in \inter K_i$ for some $j \notin \{i-1,i,i+1 \}$.
Without loss of generality, let $i=2$.
Since $o \in \inter C$, we have that the line $L_2 = R_2 \cup (-R_2)$ separates $x_1$ and $x_3$
(recall the definition of $R_i$ from the first paragraph of Section~\ref{sec:proofstarlike}).
Without loss of generality, we may assume that $x_j$ and $x_3$ lie in the same
closed half plane $H$ bounded by $L_2$, which yields that $L_3 = R_3 \cup (-R_3)$ separates $x_2$ and $x_j$.

Since $x_j \in \inter K_2$, there are points $p,q \in S$ such that $x_j \in \inter[x_2,x_2+p,x_2+q]$.
Note that by Lemma~\ref{lem:Bezdek}, we have that $x_3 \notin \inter [x_2,x_2+p,x_2+q]$.
For contradiction, suppose that $o \notin \inter [x_2,x_2+p,x_2+q]$.
Considering the cases that the line, passing through $x_2+p$ and $x_2+q$, separates $x_j$ from $o$, $x_3$ or neither,
it readily follows that at least one of $[x_2,x_2+p]$ or $[x_2,x_2+q]$ crosses both $[o,x_3]$ and the ray
emanating from $x_3$ and passing through $x_j$.
Since $\Gamma_3$ does not cross $[x_2,x_2+p]$ and $[x_2,x_2+q]$, we obtain that $x_j \in [o,x_3,w_3]$ or $x_2 \in [o,x_3,w_3]$, which,
similarly like in the proof of Lemma~\ref{lem:counterclockwise}, immediately yields that $S_j$ or $S_2$ overlaps $S_3$; a contradiction.
Hence, we obtain that $o \in \inter [x_2,x_2+p,x_2+q]$.
Without loss of generality, we may choose our notation so that $q \in H$, which implies that $R_3$ crosses the segment $[x_2,x_2+q]$.

\begin{figure}[ht]
\includegraphics[width=0.8\textwidth]{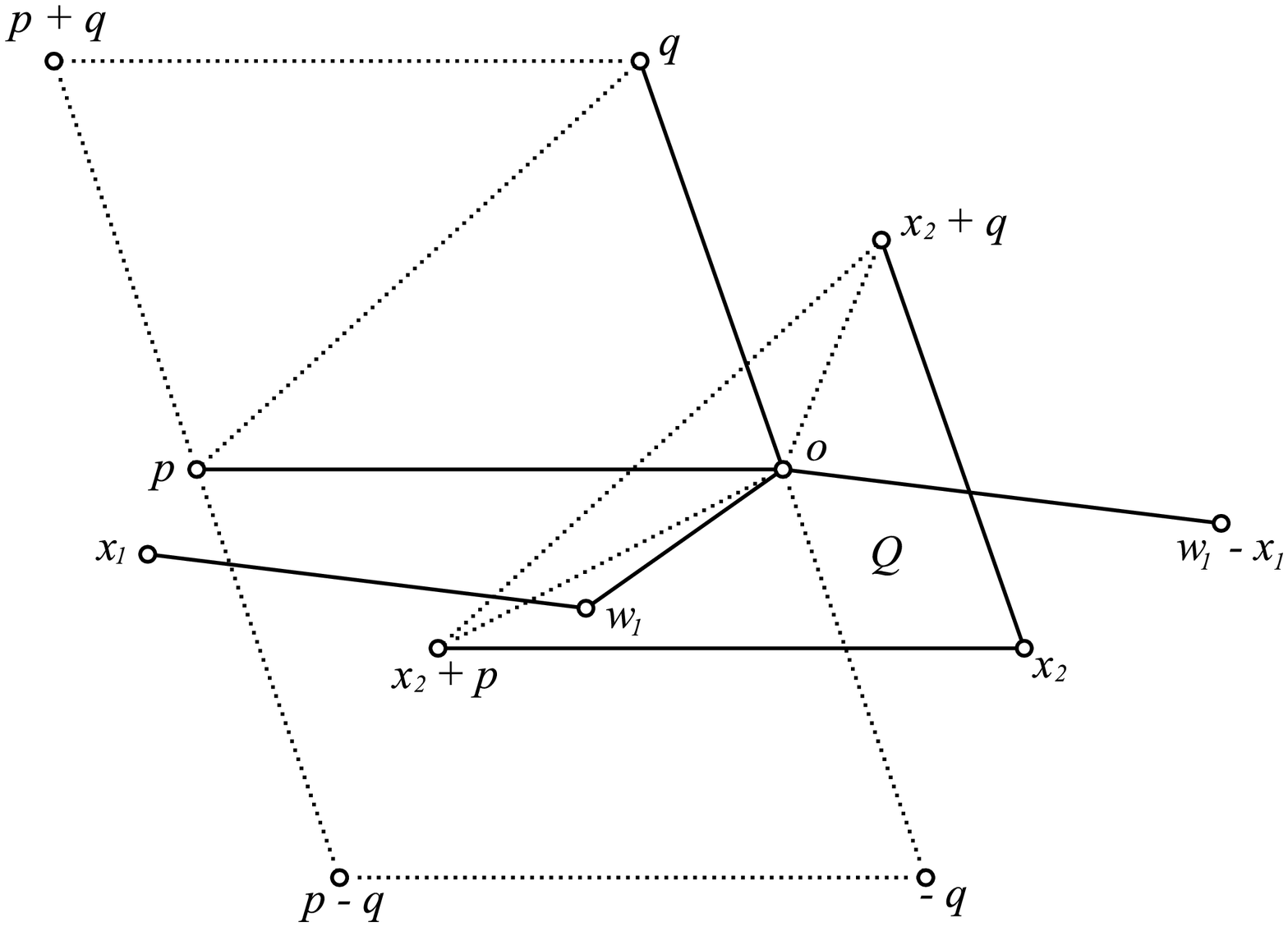}
\caption[]{}
\label{fig:lemma4}
\end{figure}

Let $Q = \inter \left( [x_2,o,x_2+p] \cup [x_2,o,x_2+q] \right)$, and consider the case that $w_1 \in Q$.
Since $S_1$ and $S$ do not overlap, we have that $x_1 \notin \inter [q,p+q,-q,p-q]$.
By Lemma~\ref{lem:Bezdek}, $x_1 \notin \inter [x_2,x_2+p,x_2+q]$, from which it readily follows that
$x_1 = \alpha p + \beta q$ with $\alpha \geq 1$.
Thus, the segment $[x_2,x_2+q]$ crosses $[o,w_1-x_1]$ (cf. Figure~\ref{fig:lemma4}).
As $[o,w_1-x_1] \subset S$, it follows that $S$ and $S_2$ overlap; a contradiction.
We may show similarly that $w_3 \notin Q$.

We obtained that, from $x_j \in \inter K_2$ and $x_j \notin \{x_1, x_2, x_3 \}$, it follows that
the angle $\angle(w_1,o,w_3)$ measured from $[o,w_1]$ to $[o,w_3]$ in the counterclockwise direction is strictly greater than $\pi$.
Note that $\angle (w_k,o,w_{k+2}) \leq \pi$ if $k \notin \{ n, 1, 2\}$, and
that $\angle(w_n,o,w_2) \leq \pi$ or $\angle(w_2,o,w_4) \leq \pi$.
Thus, $\angle (w_{i-1},o,w_{i+1})  > \pi$ holds for at most two values of $i$,
and hence the assertion immediately follows.
\end{proof}

Now we are ready to prove Theorem~\ref{thm:starlike}.
In the proof we use the following notion.
Let $D \subset \Re^2$ be a convex disk, and let $p,q \in \Re^2$.
Let $r,s \in D$ such that $[r,s]$ and $[p,q]$ are parallel, and there are no points $r',s'$ in $D$
such that $[r',s']$ is parallel to $[p,q]$ and $[r',s']$ is longer than $[r,s]$.
The \emph{$D$-distance} of $p$ and $q$ (cf. \cite{L91}) is defined as
\[
\dist_D(p,q)=\frac{2||p-q||}{||r-s||}.
\]
It is easy to see that $\dist_D(p,q)$ is equal to the distance of $p$ and $q$ in the norm
the unit disk of which is the central symmetral $\frac{1}{2}(D-D)$ of $D$.
We observe that, for any two convex disks $D \subset D'$ and points $p,q \in \Re^2$, we have
$\dist_{D'}(p,q)\leq \dist_D(p,q)$.

Note that since $S_i$ touches $S$ for every $i$, $K$ and $x_i+K$ intersect.
Thus, $\dist_K(o,x_i) = \dist_{\bar{K}}(o,x_i)\leq 2$, where $\bar{K}=\frac{1}{2}(K-K)$.
In other words, we have $X \subset 2\bar{K}=K-K$, which yields
that $\dist_C(x_i,x_j)\geq \dist_{2\bar{K}}(x_i,x_j)=\frac{1}{2}\dist_{\bar{K}}(x_i,x_j)$
for any $i \neq j$.

By Lemma~\ref{lem:nonoverlapping}, we may choose a subfamily $\F'$ of at least $\lfloor \frac{n-2}{2} \rfloor$ pairwise
separated elements of $\F$.
Let $X'$ denote the set of the translation vectors of the members of $\F'$.
Note that if $u+S$ and $v+S$ are separated, then $u+\frac{1}{2}K$ and $v+\frac{1}{2}K$ are nonoverlapping.
In other words, we have $\dist_K(u,v)\geq 1$ for any distinct $u, v \in X'$, which yields that
$\dist_{\bar{C}}(u,v) = \dist_C(u,v) \geq \frac{1}{2}$, with $\bar{C}=\frac{1}{2}(C-C)$.

Go{\l }{\c a}b \cite{G32} proved that the circumference
of every centrally symmetric convex disk measured
in its norm is at least six and at most eight (for a more accessible reference, cf. \cite{S67}).
F\'ary and Makai \cite{FM82} proved that, in any norm,
the circumferences of any convex disk $C$ and its central symmetral $\frac{1}{2}(C-C)$ are equal.
Thus, the circumference of $C$ measured in the norm
with unit ball $\bar{C}$ is at most eight.
Since $X'$ is a set of points in $\bd C$ at pairwise $\bar{C}$-distances at least $\frac{1}{2}$, we have
$\lfloor \frac{n-2}{2} \rfloor \leq \card X' \leq 16$, from which the assertion immediately follows.

\section{Proof of Theorem~\ref{thm:pocket}}

Let $K = \conv J$ and $H(J) = n$.
Consider a family $\F = \{ J_i: i=1,2,\ldots,n\}$ of pairwise nonoverlapping translates
such that each member $J_i=x_i+J$ of $\F$ touches $J$, and set $K_i = x_i + K$.
If $J = K$, then our result follows from a result of Gr\"unbaum in \cite{G61}.
Thus, we examine the case that $J$  has exactly one pocket.
Note that by \cite{HLS59}, $n \geq 6$, and if $J$ is parallelogram-like, then $n \geq 8$.
Thus, it is sufficent to prove that $n \leq 8$, and if $J$ is not parallelogram-like, then $n \leq 6$.

We choose the indices of the elements of $\F$ in a way that
$J_1 \cap J, J_2 \cap J, \ldots, J_n \cap J$ are in counterclockwise order on $\bd J$.
Let $p$ and $q$ denote the endpoints of the longest segment in $\bd K$ that contains $(\bd K) \setminus J$.

\emph{Case 1}, there is no chord of $K$, parallel to $[p,q]$, that is longer than $[p,q]$.
Then $K$ has two distinct parallel supporting lines, passing through $p$ and $q$, respectively.
Since the Hadwiger number of $J$ does not change under an affine transformation, we may assume that
$p=o$, $q=e_x$, and that the $y$-axis and the line $x=1$ support $K$.

Consider a translate $x+J$ of $J$ that overlaps $K$ but not $J$.
We show that there is no other translate of $J$ that touches $J$, overlaps $K$, and does not overlap $x+J$.
Suppose for contradiction that $y+J$ touches $J$, and it overlaps $K$ but not $x+J$.
Then $x+K$ and $y+K$ do not overlap, as otherwise $x+((\bd K) \cap J)$ and $y+((\bd K) \cap J)$ cross.
Without loss of generality, we may assume that $x+K$ and $y+K$ touch each other, $[p,q] \cap (x+K)$ is closer to $p$ than
$[p,q] \cap (y+K)$, and the $y$-coordinate of $x$ is not greater than the $y$-coordinate of $y$.
Let $[u,v]=[p,q]\cap (x+K)$, where the notation is chosen so that $u$ is closer to $p$ than $v$.
First, observe that $v+(q-p)$ is a point of $x+(q-p)+K$, and that $||v+(q-p)|| > ||q-p||$.
Furthermore, by the convexity of $K$, for any $z \in [x+q-p,v]$,
$z+K$ has a point, on the positive half of the $y$-axis, not closer to $p$ than $v+q-p$.
Thus, as $y$ is a point of the closed arc of $\bd (x+K)$ between $x+q-p$ and $v$ that does not contain $x$,
$y+K$ also has a point on the positive half of the $y$-axis, farther from $p$ than $||q-p||$, which clearly contradicts
our assumptions that $y+J$ overlaps $K$ and does not overlap $J$.
Hence, we have obtained that there is at most one member of $\F$ that overlaps $K$.
We may show similarly that there is at most one member $J_i$ of $\F$ such that $K_i$ overlaps $J$.

Let $L$ be the supporting line of $K$, parallel to and not containing $[p,q]$.
Let $[r,s] = L \cap J$, and note that $[r,s]$ may degenerate to a single point.
We choose our notation in a way that $p$, $q$, $r$ and $s$ are in counterclockwise order on $\bd J$.
Let $H^+$ be the closed half plane with $[p,q] \subset \bd H^+$ and $K \subset H^+$, and let $H^-=\Re^2 \setminus H^+$.
Let $\Gamma_p$ denote the open arc of $\bd J$, with endpoints $p$ and $s$, that does not contain $q$,
and let $\Gamma_q$ denote the open arc of $\bd J$, with endpoints $q$ and $r$, that does not contain $p$.
Clearly, if $[p,q,r,s]$ is a parallelogram, then $K$ is a parallelogram and no two translates of $K$ overlap, and thus, the assertion follows immediately from \cite{G61}.
Hence, in the following, we deal with the case that $[p,q,r,s]$ is not a parallelogram.

Consider the case that $x_i+q, x_{i+1}+q \in \Gamma_p \cup [s,r)$ for some value of $i$.
Then, clearly, $x_i, x_{i+1} \in (p-q+K)=-q+K$ (cf. Figure~\ref{fig5}), which yields that $x_i \in K_{i+1}$.
On the other hand, since $J_i$ and $J_{i+1}$ do not overlap, we have that $x_i \notin \inter K_{i+1}$,
This implies that $-q$, $x_{i+1}$ and $x_i$ are collinear, which yields that $J$ is not a disk; a contradiction.
Thus, we obtain that $x_i+q \in \Gamma_p \cup [s,r)$ for at most one value of $i$, and, by a similar argument, that $x_i \in \Gamma_q \cup [r,s)$
for at most one value of $i$.

\begin{figure}[ht]
\includegraphics[width=0.75\textwidth]{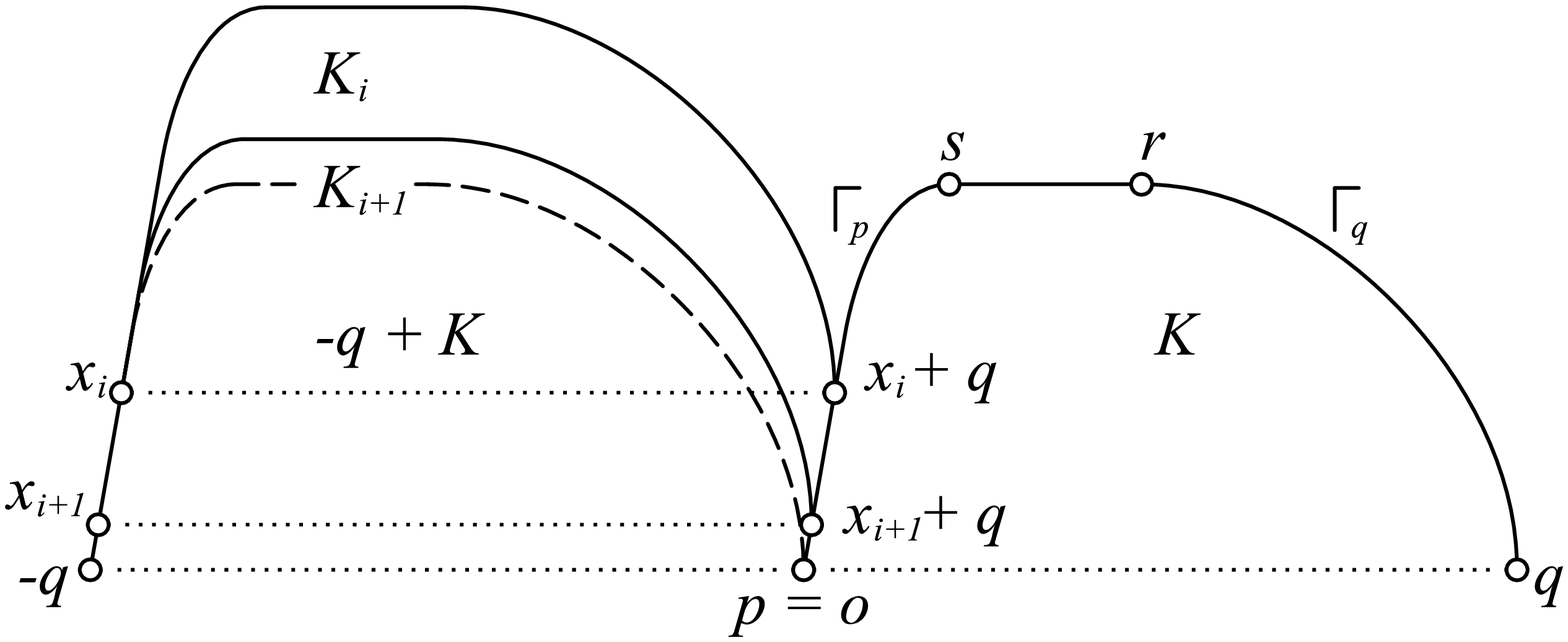}
\caption[]{}
\label{fig5}
\end{figure}

Note that for any translate $J_i \subset H^+$, at least one of the following holds:
$x_i = -q$, $x_i = q$, $x_i +q \in \Gamma_p \cup [s,r)$, $x_i \in \Gamma_q \cup [r,s)$ or $[r,s] \subset K_i$.
Hence, by the second and fourth paragraphs in Case 1, we obtain that $H^+$ contains at most five members of $\F$.
Furthermore, we observe that any two members of $\F$, that have a point in $H^-$, have nonoverlapping convex hulls,
and thus, it may hold for at most three members of $\F$.
Thus, $\card \F \leq 8$.

To finish the proof in Case 1, we examine the case that $J$ is not parallelogram-like,
and show that then $\card \F \leq 6$.
First, we show that at most four elements of $\F$ intersect $H^+$.
Suppose for contradiction that it is not so, or in other words, that for some value of $i$ we have
$J_{i-2}=q+J$, $x_{i-1} \in \Gamma_q \cup [r,s)$, $[r,s] \subset K_i$, $x_{i+1}+q \in \Gamma_p \cup [s,r)$ and $J_{i+2} = -q+J$.

Since $J_i$ and $J_{i+1}$ do not overlap, we obtain that $x_{i-1},x_{i+1} \notin (r,s)$.
If $x_{i+1}+q = s$ and $x_{i-1}=r$, then $J$ is parallelogram-like, and hence, we have that
$x_{i+1}+q \neq s$ or $x_{i-1} \neq r$, say $x_{i+1}+q \neq s$.
Then, as $J_{i+1}$ and $J_{i+2}$ do not overlap, we obtain that $x_{i+1}+q+J$ touches $J$ at $x_{i+1}+q$.
Furthermore, if $x$ is on the closed arc of $\Gamma_p \cup \{ s \}$ between $x_{i+1}+q$ and $s$, then
$x+J$ does not overlap $J$, as the region $x+K \cap K$ strictly decreases, in terms of containment, when we
move $x$ away from $x_{i+1}+q$.
This yields that $x_i$ is a point of this closed arc.
On the other hand, $x+K$ touches $J_{i-2} = q+J$ at $x+q$.
Thus, $J_{i-1}$ touches $J$ only if $x_i+q=x_{i-1}$, from which it immediately follows
that $J$ is parallelogram-like; a contradiction.

We observe that the convex hulls of any two elements of $\F$, having a point in $H^-$, are nonoverlapping,
and that at most three elements of $\F$ have points in $H^-$.
Thus, to show that $\card \F \leq 6$, it is sufficient to examine the case that, for some value of $i$,
each of $J_{i-1}$, $J_i$ and $J_{i+1}$ has a point in $H^-$.

Let $L_0$ and $L_1$ denote the $y$-axis and the line $x=1$, respectively.

\emph{Subcase 1.1}, $L_0 \cap \inter J_{i-1} \neq \emptyset$ or $L_1 \cap \inter J_{i+1} \neq \emptyset$.
Without loss of generality, let $L_0 \cap \inter J_{i-1} \neq \emptyset$, which yields that
$p$ is a common point of $J$ and $J_{i-1}$.

Consider the case that the $y$-coordinate of $x_i$ is not smaller than that of $x_{i-1}$.
Note that, by an argument similar to those used in the previous paragraphs,
if $J_i$ is contained in $K \cup (q+K_{i-1})$, then $x_i = q+x_{i-1}$, and thus, $J_i$ has a point in the open half plane $x > 1$.
Since it clearly follows also if $J_i \not\subset K \cup (q+K_{i-1})$, and it implies that $J_{i+1}$ has no point in $H^-$, which is a contradiction, we obtain that the $y$-coordinate of $x_i$ is smaller than that of $x_{i-1}$.
Similarly, if $q \in J_i$, then $J_{i+1}$ has no point in $H^-$, and thus, it follows that $q \notin J_i$.

Next, consider the case that $J_{i-2}=-q+J$.
Observe than then $q+J_{i-1}$ touches $J$. Thus, since the $y$-coordinate of $x_i$ is smaller than that of $x_{i-1}$,
any point of $J_i$, with nonnegative $y$-coordinate, is contained in $q+K_{i-1}$.
Thus, $J_i$ touches $J$ at a boundary point $u$ of $q+K_{i-1}$, which is clearly a boundary point of also $K_i$.
Since $x_i-(x_{i-1}+q)$ translates $J_i$ to $q+J_{i-1}$, this vector moves $u$ to another boundary point $u'$ of $q+J_{i-1}$; or in other words, $(x_{i-1}+q)-x_i$ moves a boundary point $u'$ of $q+J_{i-1}$ to $u$.
On the other hand, the translation by $x_{i-1}+q-x_i$ does not move any point of $[u,u']$ to a point of $K$ outside $q+K_{i-1}$.
Hence, $u \in [p,q]$, which, as $q \notin J_i$, yields that $u \in (p,q)$ and $[p,u] \subset J$.

We obtained that there is a point $u \in (p,q)$ with $[p,u] \subset J$.
We show that it yields $\card \F \leq 6$.
Observe that $q \in J_{i+1}$.
Clearly, if $q \in [x_{i+1}+r,x_{i+1}+s]$, then $J_i$ does not touch $J$.
Thus, $\inter J_{i+1} \cap H^+ \neq \emptyset$, which, by $[p,u] \subset J$, implies that $q+J \notin \F$.

Now we have $\card \F \leq 7$. Then we can have $\card \F > 6$ only if $x_{i-3}+q \in \Gamma_p \cup [s,r)$,
$[r,s] \subset J_{i+3}$ and $x_{i+2} \in \Gamma_q \cup [r,s)$.
Since $[p,u] \subset J$, we have that $x_{i-3}+q = s$, and $x_{i+3} = s$.
As the point $s+q$ is on $\bd (q+J)$, and $s+q \notin \inter K_{i+2}$,
we have that $x_{i+2}$ is in the boundary of both $J$ and $q+J$.
Thus, $x_{i+2}$ is on the line $x=1$, and $s$ is on the $y$-axis, which implies that $x_{i+1}$ is on the line $L_1$,
and that $L_1 \cap J_{i+1}$ is a translate of $[p,s]$.
On the other hand, the distance between $x_{i+2}$ and the closest point of $K_i \cap K_1$ is clearly less than that between $p$ and $s$; a contradiction.

We have shown that $-q + J \in \F$ yields that $\card \F \leq 6$.
Thus, to show that $\card \F \leq 6$,
it is sufficient to consider the case that $x_{i+2}=q$, $x_{i+3} \in \Gamma_q \cup [r,s)$,
$[r,s] \subset J_{i-3}$, $x_{i-2}+q \in \Gamma_p \cup [s,r)$.

If $L_1 \cap \inter J_{i+1} \neq \emptyset$, then we may apply an argument similar to that in the previous paragraphs.
Hence, we obtain that $L_1 \cap \inter J_{i+1} = \emptyset$, which yields that $x_{i+1}$ is on $L_1$, and that
$(x_{i+1},q) \subset (x_{i+1} + \Gamma_p)$.
Since $L_0 \cap \inter J_{i-1} \neq \emptyset$ and $J_i$ touches $J$, we obtain that
the $y$-coordinate of $x_i$ is less than that of $x_{i+1}$,
and that there is a point $u \in (p,q)$ such that $[u,q] \subset J$.
Thus, $x_{i+3} \notin \Gamma_q$, from which we have that $x_{i+3} = r$ and $x_{i-3}=r-q$.
But this implies that $J_{i-2}$ does not touch $J$; a contradiction.

\emph{Subcase 1.2}, $L_0 \cap \inter J_{i-1} \cap = \emptyset$ and $L_1 \cap \inter J_{i+1} = \emptyset$.
In this case $J \cap L_1 = [w,q]$ and $L_0 \cap J = [z,p]$ for some points $w$ and $z$ with $z \neq p$ and $w \neq q$.
Observe that if neither $q+J$ nor $-q+J$ belongs to $\F$, then $\card \F \leq 6$.
Hence, it suffices to consider the case that at least one of them, say $-q+J$, belongs to $\F$, which yields that $x_{i-1}=-z-q$, and $|| w-q|| \geq || z-p||$.

Note that $|| w-q|| > || z-p||$, as otherwise $J$ is parallelogram-like.
Since $q+J \in \F$ yields that $||w-q|| = || z-p||$, we have also that $q+J \notin \F$.
Thus, if $\card \F > 6$, then $x_{i-3}+q \in \Gamma_p \cup \{ s \}$, $[r,s] \subset K_{i+3}$
and $x_{i+2} \in \Gamma_q \cup \{ r \}$.
Without loss of generality, we may assume that $x_{i+1}=q-z$.
Since $J$ is not parallelogram-like, $L_0 \cap \inter J_{i-3} \neq \emptyset$.
Hence, since $J_{i+3}$ touches $J$, $L_0 \cap \inter J_{i+3} \neq \emptyset$, which yields that
$L_1$ supports $J_{i+2}$.
But then $w \in J_{i+2}$, and $J_{i+2}$ and $J_{i+3}$ overlap; a contradiction.
This finishes the proof in Case 1.

\emph{Case 2}; there is a chord in $K$, parallel to $[p,q]$, which is longer than $[p,q]$.
Clearly, in this case $J$ is not parallelogram-like, and hence, we need to prove that $\card \F \leq 6$.
Since the proof is similar to that of Case 1, we just sketch it.

First, we prove the following lemma that we need in the proof. We note that the proof of this lemma is included in the proof of Theorem 7 in \cite{JL07}.

\begin{lem}\label{lem:parallelogram}
Let $K$ be a convex disk. If $\bar{K}=\frac{1}{2}(K-K)$ is a parallelogram, then $K$ is a translate of $\bar{K}$.
\end{lem}

\begin{proof}
First, observe that $\bar{K}=\frac{1}{2}(K-K)$ if and only if $K$ is a convex disk of constant width two in the norm of $\bar{K}$.
Let $\bar{K} = [a_1, a_2, a_3, a_4]$ be a parallelogram, where $a_1$, $a_2$, $a_3$ and $a4$ are in this counterclockwise order in $\bd \bar{K}$.
Let $L_1$ and $L_2$ be the supporting lines of $K$ parallel to $[a_1,a_2]$
such that the translate of $L_1$ by $a_4-a_1$ is $L_2$.
Let $[b_1, b_2] = K \cap L_1$ and $[b_3,b_4] = K \cap L_2$ such that $b_2-b_1$
and $b_3-b_4$ are positive multiples of $a_2-a_1$.
Let $c_3 = b_1 + (a_3-a_1)$ and $c_4 = b_2 + (a_4-a_2)$.
Since $K$ is of constant width two, we have $[c_3, c_4] \subset [b_3, b_4]$.
Observe that $||c_4-c_3|| = 2||a_2-a_1||-||b_2-b_1||$.
As $||b_4-b_3|| \leq ||a_2-a_1||$, this implies that $||b_2-b_1|| = ||b_4-b_3|| = ||a_2-a_1||$,
$c_3 = b_3$ and $c_4 = b_4$.
Hence $K' = [b_1,b_2,b_3,b_4]$ is a translate of $\bar{K}$.
As $K' \subset K$ and $K$ is a convex disk of constant width two in the norm of $\bar{K}$, we have $K' = K$. 
\end{proof}

Now we return to the proof of our theorem.
Suppose for contradiction that $\card \F \geq 7$.
We leave it to the reader to show, by methods similar to those used in Case 1, that
the convex hulls of the elements of $\F$ are mutually nonoverlapping (though they may overlap $K$),
and that the points $x_1, x_2, \ldots, x_n$ are in convex position.
Let $\Theta$ denote the closed polygonal curve with endpoints $x_1, x_2, \ldots, x_n$ in counterclockwise order.
Observe that $\Theta \subset K-K$, and the sides of $\Theta$ are of at least unit length in the norm defined by the
difference body $K-K$ of $K$.
By Theorem 2 of \cite{L04}, there is a convex $n$-gon $P$, inscribed in $K-K$, such that the sides of $P$ are of at least unit length in the norm of $K-K$.
In other words, there are $n$ mutually nonoverlapping translates of $\bar{K}=\frac{1}{2}(K-K)$, each of which touches $\bar{K}$.
Clearly, $K$ is not a parallelogram.
Thus Lemma~\ref{lem:parallelogram} yields that $\bar{K}$ is not a parallelogram,
and hence, by \cite{G61}, we have that $n \leq H(\bar{K}) = 6$; a contradiction.

\noindent
{\bf Acknowledgement.}
The author would like to express his sincere gratitude to an anonymous referee and to M\'arton Nasz\'odi
for their many suggestions to improve the quality of this paper.


\begin{thebibliography}{99}

\bibitem{B97} A. Bezdek,
	\emph{On the Hadwiger number of a starlike disk},
	in: Intuitive Geometry (Budapest,1995), Bolyai Soc. Math. Studies {\bf 6} (1997), 237-245.
\bibitem{BKK95} A. Bezdek, K. Kuperberg and W. Kuperberg,
	\emph{Mutually contiguous translates of a plane disk},
	Duke Math. J. {\bf 78} (1995), 19-31.
\bibitem{BMP05} P. Brass, W. Moser and J. Pach,
	\emph{Research Problems in Discrete Geometry},
	Springer, New York, 2005.
\bibitem{CL07} O. Cheong, M. Lee,
	\emph{The Hadwiger number of Jordan regions is unbounded},
	Discrete Comput. Geom. {\bf 37} (2007), 497-501.
\bibitem{FM82} I. F\'ary and E. Makai, Jr.,
	\emph{Isoperimetry in variable metric},
	Stud. Sci. Math. Hungar. {\bf 17} (1982), 143-158.
\bibitem{G32} S. Go{\l }{\c a}b,
	\emph{Some metric problems of the geometry of Minkowski} (in Polish),
	Trav. Acad. Mines Cracovie {\bf 6} (1932), 1-79.
\bibitem{G61} B. Gr\"unbaum,
	\emph{On a conjecture of H. Hadwiger},
	Pacific J. Math. {\bf 11} (1961), 215-219.
\bibitem{HLS59} C. J. A. Halberg Jr., E. Levin and E. G. Straus,
	\emph{On contiguous congruent sets in Euclidean space},
	Proc. Amer. Math. Soc. {\bf 10} (1959), 335-344.
\bibitem{JL07} A. Jo\'os and Z. L\'angi,
	\emph{On the relative distances of seven points in a plane convex body}
	J. Geom. {\bf 87} (2007), 83-95.
\bibitem{L08} Z. L\'angi,
	\emph{On the Hadwiger numbers of centrally symmetric starlike disks},
	Beitr\"age Algebra Geom. {\bf 50}(1) (2009), 249-157.
\bibitem{L91} M. Lassak,
	\emph{On five points in a plane convex body pairwise in at least unit relative distances},
	Intuitive geometry (Szeged, 1991), 	Coll. Math. Soc. J\'anos Bolyai {\bf 63},
	North Holland, Amsterdam, 1994, 245-247. 
\bibitem{L04} M. Lassak,
	\emph{On relatively equilateral polygons inscribed in a convex body},
	Publ. Math. Debrecen {\bf 65} (2004), 133-148.
\bibitem{S67} J. J. Sch\"affer,
	\emph{Inner diameter, perimeter and girth of spheres},
	Math. Ann. {\bf 173} (1967), 59-79.
\end{thebibliography}
\end{document}